\documentclass{amsart}

\usepackage{eufrak}

\usepackage{amssymb,amsmath, amscd, mathrsfs, yfonts, bbm, graphics, dsfont}

\title[] {On fluctuations of traces of large matrices over a non-commutative algebra}

\author{Mihai Popa and Yong Jiao}

\address{
Department of Mathematics and Statistics, Queen's University, 
Jeffery Hall, Kingston, Ontario, K7L 3N6, Canada, and
\newline
Institute of Mathematics �Simion Stoilow� of the Romanian Academy, P.O. Box 1-764,
Bucharest, RO-70700, Romania}
\email{popa@mast.queensu.ca}

\address{
Institute of Probability and Statistics, Central South University, Changsha 410075, China}
\email{jiaoyong@csu.edu.cn}
\thanks{This work was supported by the National Natural Science Foundation of China (grant 11001273), and by the  Research Fund for International
Young Scientists of NSFC (grant 11150110456)}

\newtheorem{claim}{}[section]
\newtheorem{defn}[claim]{Definition}
\newtheorem{thm}[claim]{Theorem}
\newtheorem{lemma}[claim]{Lemma}
\newtheorem{remark}[claim]{Remark}
\newtheorem{prop}[claim]{Proposition}

\newcommand{\cA}{\mathcal{A}}
\newcommand{\cH}{\mathcal{H}}

\newcommand{\cB}{\mathcal{B}}

\newcommand{\cE}{\mathcal{E}}

\newcommand{\lra}{\longrightarrow}

\newcommand{\Tr}{\text{Tr}}
\newcommand{\tr}{\text{tr}}

\newcommand{\wpi}{\widehat{\pi}}

\usepackage{amssymb}
\input xy
\xyoption{all}

\begin{document}

\maketitle

\begin{abstract}
 The paper investigates the asymptotic behavior of (non-normalized) traces of certain classes of  matrices 
with non-commutative  random variables as entries. We show that, unlike in the commutative framework,
the asymptotic behavior of matrices with free circular, respectively with Bernoulli distributed Boolean independent entries
is described in terms of free, respectively Boolean cumulants. We also present an exemple of relation of monotone independence arising from the study of Boolean independence.
\end{abstract}

%%%%%%%%%%%%%%%%%%%%%%%%%%%%%%%%%%%%%%%%%%%%%%%%%%%%%%%%%%%%%%%%%%%%%%%%%%%5
%%%%%%%%%%%%%%%%%%%%%%%%%%%%%%%%%%%%%%%%%%%%%%%%%%%%%%%%%%%%%%%%%%%%%%%%%%%%%%%%%%
%%%%%%%%%%%%%%%%%%%%%%%%%%%%%%%%%%%%%%%%%%%%%%%%%%%%%%%%%%%%%%%%%%%%
%%%%%%%%%%%%%%%%%%%%%%%%%%%%%%%%%%%%%%%%%%%%%%%%%%%%%%%%%%%%%%%%%%%%%%%%%%
\section{Introduction}

The fluctuations of traces of various classes of random matrices have been studied in the last two decades
in both physics (see, for example \cite{bs}, \cite{fmp}, \cite{kkp}) and mathematics (see \cite{mingo-speicher}, \cite{mss} \cite{mingo-p-r})  literature. Extensive works (such as \cite{vdn}, \cite{speichernica}) indicate that free independence is best suited to describe the interaction of important classes of independent ensembles of random matrices  with respect to normalized traces.  It was shown that free independence and the corresponding Central Limit Theorem laws  (centered semicircular distributions) behave in a very regular manner when tensoring with algebras of complex matrices (\cite{speicher-hab}). 
In order to address interactions of independent ensembles of random matrices with respect to  \emph{unnormalized} traces
 (fluctuation moments, higher order trace-moments), recent works, such as \cite{mingo-speicher} and \cite{mss}, introduced the notion of second order freeness or the more refined real second order  freeness (\cite{mingo-p-r}). The  present paper comes as an addendum to these works, more in the spirit of \cite{oyvind}. More  precisely, while  \cite{mingo-speicher}, \cite{mss} and \cite{mingo-p-r} study the behavior of important classes of random matrices  with entries in a commutative algebra, we present some similar results for the case when the entries are not commuting.

The results presented here bring  contributing evidence to the special nature  of second order  independence relations. We show that although ensembles of self-adjoint Gaussian random matrices can be well approximated at first order level
by ensembles of matrices with free semicircular entries, the second order  behavior of these two classes is different.
Also, classical cumulants are well-suited to describe higher order independence relations of ensembles of random matrices with commuting, independent entries; the results from Section \ref{section:free}, respectively Section \ref{section:boolean}, seem to  indicate that free, respectively Boolean cumulants   are appropriate  to describe higher order independence relations of ensembles of random matrices with free, respectively Boolean independent entries.

In what follows, the paper is organized in 3 sections. Section 2 contains some preliminary notions and results on permutations and partitions of an ordered set and non-commutative notions of independence. 

 Section \ref{section:free} presents some results in the  study of higher order behavior of ensembles of self-adjoint matrices with free circular entries. We first show (see Theorem \ref{thm:1}) that the free cumulants of unnormalized traces of such ensembles have a very similar behavior to the results presented in \cite{mss} concerning classical cumulants of ensembles of random matrices with independent Gaussian entries. We also show that,  in this framework, a substitute for second order freeness  from \cite{mingo-speicher} is  Property ($\ast$) that we define  in Section \ref{section:ast}.

The shorter Section \ref{section:boolean} presents  some results concerning Boolean independence. This non-unital notion of non-commutative independence (\cite{sbg}) is by far less studied than freeness, but it was shown to be of relevance in some problems from Theoretical Physics (\cite{wvw}), Free Probability (\cite{mp-jcw}), completely positive maps (\cite{mp-vv}) or Real Analysis (\cite{b}). We show that the Boolean cumulants of traces of ensembles of self-adjoint matrices with Bernoulli distributed 
boolean independent entries and constant matrices have a similar behavior to the classical cumulants of traces of Gaussian ensembles, as presented in \cite{mss}, respectively to the free cumulants of traces of semicircular ensembles, as presented in  Section \ref{section:free}. In addition, Theorem \ref{thm:monotone} presents a new exemple of monotone independence relation, here arising from the relations between constant matrices and matrices with Bernoulli distributed Boolean independent entries.

\section{Preliminaries}
\subsection{Partitions on an ordered set}
 For a positive integer $ n $, we will denote by $ [ n ] $ the ordered set $ \{ 1, 2, \dots, n \}$. 
By a \emph{partition}  $ \pi $ on $ [n ] $ we will understand a family 
$ B_1, B_2, \dots, B_{ q(\pi) } $ 
of pairwise disjoint nonvoid subsets of $ [ n ] $, called \emph{blocks} of $ \pi $, such that 
$ \cup_{l=1}^{ q( \pi ) } B_l = [ n ] $. If each block of $ \pi $ has exactly 2 elements, then $ \pi $ will be called a \emph{pairing}.  
The set of all partitions, respetively pairings on $ [ n ] $ will be denoted by 
$ P(n) $, respectively by $ P_2(n) $. 

The set $ P(n ) $ is a lattice under the partial order relation $ < $,  given by $ \sigma < \pi $ if any block of 
$ \sigma $ is contained in some block of $ \pi $. The maximal element of the lattice is $ 1_n $, the partition with a single block. For $ \pi, \sigma \in P(n) $, define
\[
\sigma \vee \pi = \inf \{ \tau: \ \tau > \pi,\ \tau > \sigma \}.
\]

A partition $ \pi \in P(n) $ will be called \emph{non-crossing} if for any $ B ,  D $ disjoint blocks of $ \pi $, there exists no 4-tuple $ i< j <k < l $ from $ [n] $ such that $ i,k \in B $ and $ j, l \in D $. The sets of all non-crossing partitions, respectively non-crossing pair-partitions of $ [ n ] $ 
will be denoted by $ NC(n) $, respectively $ NC_2(n) $.

A partition $ \pi \in P(n) $ will be called \emph{interval partition} if each block of $ \pi $ contains only consecutive elements from $ [ n ] $. We will denote the set of all interval partitions, respectively pairings of $ [ n ] $ by $ I(n) $, respectively $ I_2(n) $. Note that if $ n $  is odd, then $ I_2(n) = \emptyset $; if $ n $ is even then $ I_2(n) $ has only one element, namely the partition of blocks $\{ ( 2 k - 1 , 2 k )  :  1 \leq k \leq \frac{n}{2} \} $.

A permutation $ \gamma\in  S_n $ (the Symmetric group of order $ n $) will be uniquelly identified with a partition
on $ [ n ] $ by taking the blocks to equal (as sets) the cycles. A pair partition $ \pi \in P_2 (n) $ can be uniquelly identified with a permutation from $ S_n $ by taking the cycles to equal the blocks of $ \pi $.  The following result connecting partitions and permutations was proved in \cite{mingo-nica} (see \cite{mingo-nica}, relation 2.9):
\begin{prop}\label{part-perm}
If $ \tau, \sigma \in S_n $ , then
\[ 
\#( \tau ) + \#(\tau^{ -1} \sigma ) + \# ( \sigma ) \leq n+ 2 \#( \tau \vee \sigma )
\]
where in the left hand side of the equation $ \tau, \sigma $ are seen as permutations and in the 
right hand side as partitions.
\end{prop}

For $ \sigma \in S_m $ and $ A_1, \dots, A_m $ some $ N \times  N $ complex matrices, we will define $ \Tr_\sigma( A_1, \dots, A_m) $ as follows. If $ \sigma $ has the cycle decomposition 
\[
 \sigma = \prod_{ q =1 }^ n \left( i( q, 1), i( q , 2 ) , \dots, i( q, l( q ) ) \right),
\]
 then we define
\[
 \Tr_\sigma( A_1, \dots, A_m)  = \prod_{ q =1 }^n  \Tr( A_{ i( q, 1) } \cdot A_{  i( q , 2 ) } \cdots A_{i( q, l( q ) ) }.
\]

If $ \overrightarrow{j} = ( j_1, \dots, j_m ) $ is a multitindex and $ \sigma \in S_m $, we will write that $ \overrightarrow{j}= \overrightarrow{j} \circ \sigma $  if $ \sigma(k) = l $ implies $ j_k = j_l $. 

 We will use the following version of the Lemma 5 from \cite{mingo-p-r}:
\begin{lemma}\label{lemma:2}
Suppose that $ A_k $  are $ N \times N $  complex matrices with entries $ a^{( k )  }_{ i, j } $ , where $ 1 \leq k \leq m $.
If $ \pi\in P_2 (m ) $ and $ \sigma \in S_m $ are such that $ k + \pi\sigma(k) \equiv 1 $ (mod 2) for all 
$ k \in \{ 1, \dots, m \} $, then there exists some $ \tau \in S_m $ such that
\[
\sum_{ \overrightarrow{ j } = \overrightarrow{j} \circ \pi\sigma } 
a^{( 1)}_{ j_{ \sigma ( 1) } j_{ \sigma( 2 ) } }
\cdots
a^{ ( m )}_{ j_{ \sigma ( 2m - 1 ) } j_{ \sigma ( 2 m ) } } = \Tr_{ \tau }( A_1, \dots, A_m ),
\]
and if $ ( i_1, i_2, \dots, i_q ) $ is a cycle of $ \tau $, then $ \pi\sigma ( i_v + 1 ) = i_{ v+ 1 } $.
\end{lemma}

\subsection{Non-commutative probability spaces and independence relations.}
Following \cite{speichernica}, by a non-commutative C$^\ast$-probability space we will understand a couple $(\cA, \phi)$, where $ \cA $ is a unital C$^\ast$-algebra and $ \phi: \cA \lra \mathbb{C}  $ is a positive, linear, unital map. The elements of $ \cA $ will be called \emph{non-commutative random variables}.

For $ n \geq 1 $ , the $ n$-th classical, free, respectively Boolean cumulant are the $ n$-multilinear maps from $ \cA^n$ to $ \mathbb{C} $ denoted by $ k_n $, $ \kappa_n $, respectively $ \mathfrak{b}_n $ and given by the following recurrence relations:
\begin{align}
\phi(X_1 X_2 \cdots X_n) 
= &
\sum_{ \pi \in P( n ) } \prod_{ \substack{ B \in \pi \\ B = \{ b(1),  b(2), \dots, b(s) \} } }
k_{ s } ( X_{ b(1) }, X_{ b(2 ) } , \dots, B_{ b(s ) } )\label{eq:1}\\
=&
\sum_{ \sigma \in NC( n ) } \prod_{ \substack{ B \in \sigma \\ B = \{ b(1),  b(2), \dots, b(s) \} } }
\kappa_{ s } ( X_{ b(1) }, X_{ b(2 ) } , \dots, B_{ b(s ) } )\label{eq:2}\\
=&
\sum_{ \tau \in I( n ) } \prod_{ \substack{ B \in \tau \\ B = \{ b(1),  b(2), \dots, b(s) \} } }
\mathfrak{b}_{ s } ( X_{ b(1) }, X_{ b(2 ) } , \dots, B_{ b(s ) } ).\label{eq:3}
\end{align}

 Two unital subalgebras $ \cA_1 $, $ \cA_2 $ of $ \cA $ are said to be \emph{free independent}
if 
\[ \phi(a_1 a_2 \cdots a_n ) = 0 \]
whenever $ a_i $ ($1 \leq i \leq n  $) are such that $ \phi( a_i ) = 0 $, $ a_i \in \cA_{ \epsilon(i) } $ with $ \epsilon(i) \in \{ 1, 2 \} $, $ \epsilon(i) \neq \epsilon ( i + 1 ) $. An equivalent condition (see \cite{speichernica}) is that 
$ \kappa_n(a_1, a_2, \dots, a_n ) = 0 $ 
whenever $ a_i \in \cA_{ \epsilon(i) }$
such that not all $ \epsilon(1), \epsilon(2), \dots, \epsilon(n) $ are equal.

Two subalgebras $ \cA_1$, $ \cA_2 $ of $ \cA $ are said to be \emph{Boolean independent} 
(see  \cite{wvw}, \cite{mp-bool})
if 
\[ 
\phi(a_1 a_2 \cdots a_n ) = \phi(a_1) \phi(a_2) \cdots \phi(a_n) 
\]
whenever $ a_i \in \cA_{ \epsilon( i ) } $ with $ \epsilon(i) \neq \epsilon ( i + 1 ) $.
An equivalent condition (see \cite{mp-bool}) is that 
$ \mathfrak{b}_n (a_1, \dots, a_n ) = 0 $ whenever $ a_i \in \cA_{ \epsilon( i ) } $ such that
not all $ \epsilon(1), \epsilon (2), \dots, \epsilon(n ) $ are equal.

We will say that a subalgebra $ \cA_1 $ of $ \cA $ is \emph{monotone independent}
 (see  \cite{muraki}, \cite{mp-pjm}, \cite{sbg})
from $ \cA_2 $, another subalgebra of $ \cA $ if, for all $ x_1, x_2 \in \cA $, $ b_1, b_2 \in \cA_2 $ and $ a \in \cA_1 $ we have that 
\begin{align*}
\phi(x_1 b_1 a) &= \phi(x_1 b_1) \phi(a)\\
\phi(a b_2 x_2 ) & = \phi(a) \phi(b_2 x_2 )\\
\phi(x_1 b_1 a b_2 x_2 ) &= \phi( x_1 b_1 b_2 x_2) \phi(a ).
\end{align*}

A selfadjoint element $ x \in \cA $ is said to be semicircular, respectively Bernoulli distributed of mean 0 and variance $ \sigma > 0 $ if $ \kappa_n ( x, x, \dots, x ) = \delta_n, 2 \sigma^2 $, respectively if $ \mathfrak{b}_n ( x, x, \dots, x ) = \delta_{ n, 2 } \sigma^2 $. 

 The folowing result is known as the Free Wick Theorem (see \cite{ege-mp}):

\begin{prop}\label{free:wick}
 Let $ H $ be a real Hilbert space with orthonormal basis $ \{ e_i\}_{ i \in I } $ and 
$ \varphi: H \otimes \mathbb{C} \lra \cA $
 be a linear map such that 
$ \{ \varphi( e_i ) \}_{ i \in I }$
is a free family of semicircular elements of mean 0 and variance 1. Then, for any $ f_1, \dots, f_n \in H \otimes \mathbb{C} $ we have that
\[
 \phi( \varphi(f_1) \varphi(f_2) \cdots \varphi(f_n) ) = 
\sum_{ \pi\in NC_2(n) } \prod_{ ( i, j) \in \pi } \langle f_i, f_j \rangle. 
\]
\end{prop}

A similar result holds true for the Boolean framework. More precisely, we have the following proposition.

\begin{prop}\label{wick:boolean}
 Let $ H $ be a real Hilbert space with orthonormal basis $ \{ e_i\}_{ i \in I } $ and 
$ \varphi: H \otimes \mathbb{C} \lra \cA $
 be a linear map such that 
$ \{ \varphi( e_i ) \}_{ i \in I }$
is a Boolean independent family of Bernoulli distributed elements of mean 0 and variance 1. Then, for any $ f_1, \dots, f_n \in H \otimes \mathbb{C} $ we have that
\begin{equation}\label{eq:4}
 \phi( \varphi(f_1) \varphi(f_2) \cdots \varphi(f_n) ) = \left\{
\begin{array}{ll}
         0 & \text{if $ n$ is odd };\\
       \displaystyle  \prod_{ i =1}^{ \frac{n}{2} } \langle f_{ 2i - 1}, f_{ 2 i } \rangle & \text{if $n $ is even }.\end{array} \right.
\end{equation}

Equivalently, 
\begin{align*}
 \phi( \varphi(f_1) \varphi(f_2) \cdots \varphi(f_n) ) 
= & 
\langle f_1, f_2 \rangle \phi ( \varphi( f_3) \cdots \varphi( f _n ) ) \\
=&
\sum_{ \pi\in I_2(n) } \prod_{ ( k, l ) \in \pi } \langle f_k, f_l \rangle. 
\end{align*}
\end{prop}

\begin{proof}

Since both sides of the equation are multilinear in $ f_1, \dots, f_n $, it suffices to prove the result for all $ f_k $ from the orthonormal basis $ \{ e_i \}_{ i \in I } $.  For $ n \leq 2 $, the equality follows from (\ref{eq:2}) and the fact that 
$ \phi( \varphi(e_i ) )=0 $.

 For $ n > 2 $, let 
 $ m = \max \{ p : f_k = f_1, \ 1 \leq  k \leq p \} $. From equation (\ref{eq:2}) we have that
\[ 
\phi( \varphi(f_1) \cdots \varphi( f_n ) ) = \phi ( \varphi(f_1) \cdots \varphi( f_m ) )\cdot
\phi( \varphi( f_{ m + 1 } ) \cdots \varphi( f_n ) ).
\]
If $ m $ is odd then $ \phi( \varphi(f_1)^m ) = 0 $;
 also, the right hand side of (\ref{eq:4}) will contain the factor $ \langle f_m, f_{m+1 } \rangle $ which cancels, hence in this case the equality holds. Suppose than $ m $ is even. Equation
(\ref{eq:2})  gives
\[ \phi( \varphi(f_1 )^m ) = \sum_{ \tau \in I(n) } \prod_{ B = \text{block of } \tau} 
\mathfrak{b}_{ | B | } (\varphi(f_1 ), \dots, \varphi( f_1 ) ).
\]
Since $ \varphi( f_1 ) $ is Bernoulli distributed of mean 0 and variance 1, all its Boolean cumulants cancel, except the ones of order 2, which equal 1, therefore
\begin{align*}
\phi ( \varphi(f_1) \cdots \varphi( f_m ) ) 
&= \phi ( \varphi(f_1)^m ) = 1 
= [ \phi ( \varphi ( f_1)^2 ) ]^{ \frac{m}{2} } \\
&=  \langle f_1, f_1 \rangle^{ \frac{m}{2} } 
= \prod_{ i =1}^{ \frac{m}{2} }\langle f_{2i -1}, f_{ 2i }\rangle,
\end{align*}
and the conclusion follows by induction.

\end{proof}

\subsection{Ensembles of random matrices.}
Throughout the paper, $ M_N (\mathbb{C} ) $ will denote 
 the C$^\ast$-algebra of $ N \times N $ square matrices with complex entries 
and $ M_N(\cA ) $ the C$^\ast$-algebra $M_N(\mathbb{C} ) \otimes \cA $; by a random
matrix with entries in $ \cA $ we will understand an element of $ M_N (\cA )$. 
Throughout  the paper, 
by an \emph{ensemble of random matrices} with entries in $ \cA $ we will understand a set 
$ \{ A_{ i, N }\}_{ i \in I, N \in \mathbb{Z}_+ } $ such that  $ A_{i, N } \in M_N (\cA) $
for all $ i, N $. The ensemble $ \{ A_{ i, N } \}_{ i \in I, N \in \mathbb{Z}_+ } $ is said to have
\emph{limit distribution}  if for any $ i_1, \dots, i_n \in I $, the limit
\[ \lim_{ N \lra \infty } \tr(A_{i_1, N} \cdot A_{ i_2, N } \cdots A_{ i_n, N } ) \]
exists and it is finite.

\section{Random matrices with free circular entries}\label{section:free}

\subsection{Semicircular random matrices}\label{section:31}

Let $ H $ be a real Hilbert space and $ \cH = H \otimes \mathbb{C} $ be its complexification.
Let $ \{ S_N ( f ) \}_{ N \in \mathbb{ Z }_+, f \in \cH } $ be an ensemble of random matrices such that
$ S_N ( f ) = [ c_{ i, j } ( f ) ]_{ i, j =1}^N $ with $ c_{i, j }( f ) \in \cA $  such that
\begin{enumerate}
\item[(i)] $ \displaystyle 
\phi( c_{ i, j }( f ) c_{ k, l } ( g ) ) =\frac{1}{N} \delta_{ i, k } \delta_{ j, l } \langle f, g \rangle $
\item[(ii)]$ S_N ( f )^\ast = S_N ( f ) $
\item[(iii)]$ \{ \Re c_{ i, j }( f ) , \Im c_{ i, j } ( f ) \}_{ 1\leq i\leq j \leq N } $ form a free family of semicircular elements of mean 0.
\end{enumerate}

Let $ l_1, \dots, l_ r > 0 $ and put $ l_0 = n_ 0 = 0 $, and, for $ 1 \leq k \leq r $, put 
$ n_k = n_{ k -1}+ l_ k $. Let $ m = n_r = l_1 + l_2 + \dots + l_r $; take $ f_1, f_2, \dots, f_m \in \cH $ and let
\[ 
Y^{ ( N ) }_{ k } = X_N ( f_{ n_{ k -1 } + 1 } ) X_N ( f_{ n_{ k -1 } + 2 } ) \cdots X_N ( f_{ n_{ k }  } ).
\]

\begin{thm}\label{thm:1}
With the notations above, we have that
\[
\kappa_r ( \Tr(Y^{( N )} _1), \Tr( Y^{ ( N ) }_2 ), \dots, \Tr( Y^{ (N ) }_r ) ) = O(N^{ 2 - r } ).
\]
\end{thm}

\begin{proof}
Let $ \gamma \in S_m $ with  cycle decomposition
 $\displaystyle  \prod_{ k =1}^r ( n_{ k-1} + 1, n_{ k -1 } + 2 , \dots, n_k ) $.
Then the Free Wick Theorem gives:
\begin{align*}
\phi( \Tr ( Y_1 ) \cdots \Tr( Y_r ) ) 
&=
\sum_{ i_1, \dots, i_m =1}^N 
\phi ( c_{ i_1 i_{ \gamma( 1 ) } } (f_1 ) \cdots c_{ i_m i_{ \gamma (m ) } } ( f_m ) )\\
&= 
\sum_{ i_1, \dots, i_m =1}^N 
N^{ - \frac{m}{2}}
\sum_{ \sigma \in NC_2 (m ) } 
\prod_{ ( k, l ) \in \sigma }
\langle f_k, f_l \rangle 
\delta_{ i_k, i_{ \gamma( l ) } }
\delta_{ i_l, i_{ \gamma( k ) } }\\
&=
\sum_{ \sigma \in NC_2 ( m ) }
N^{ \# ( \gamma  \sigma ) - \frac{m }{ 2 } }
\prod_{ ( k, l ) \in \sigma } \langle f_k, f_l \rangle 
\delta_{ i_k, i_{ \gamma( l ) } }
\delta_{ i_l, i_{ \gamma( k ) } }
\end{align*}

The blocks of $ \gamma \vee \sigma $ are unions of blocks of $ \gamma $. 
Suppose 
that 
$ \gamma \vee \sigma \not \in NC( m ) $, that is there exit $ B_1, B_2, B_3, B_4 $ blocks of
$ \gamma $ in this lexicographical order such that
$ B_1, B_3 \in D_1 $ and $ B_2, B_4 \in D_2 $ with $ D_1, D_2 $ distinct blocks of 
$ \gamma \vee\sigma $. Hence there exist $ b_k \in B_k $ ( $ 1 \leq k \leq 4 $ such that
$ \sigma (b_1) = b_3 $ and $ \sigma ( b_2 ) = b_4 $, which implies that 
$ \sigma \not \in NC_2(m) $.  We have then tha $ \gamma \vee \sigma \in NC(m ) $ and an inductive argument on $ r $ gives us that 
\begin{equation}\label{eq:5}
\kappa_r ( \Tr(Y^{(N)}_1), \dots, \Tr ( Y^{ ( N )}_r ) )= 
\sum_{\substack{ \sigma \in NC_2 ( m )\\ \gamma \vee \sigma =1_m } }
N^{ \# ( \gamma \sigma ) - \frac{m }{ 2 } }
\prod_{ ( k, l ) \in \sigma } \langle f_k, f_l \rangle 
\delta_{ i_k, i_{ \gamma( l ) } }
\delta_{ i_l, i_{ \gamma( k ) } }
\end{equation}
Since the number of cycles is the same in a conjugacy class of permutations and $ \sigma^2$ is the identity permutation, we have that 
$ \# ( \gamma \sigma ) = \# ( \sigma \gamma \sigma^2)
= \# ( \sigma^{ -1} \gamma ) 
$.
Also, $ \# ( \sigma ) = \frac{m}{2} $ and $ \# (\gamma ) = r $, hence, for $ \gamma \vee \sigma = 1_m $, Proposition \ref{part-perm} gives
\[
 r + \# ( \gamma\sigma ) +\frac{m}{2} \leq m + 2
\]
which implies $ \# ( \gamma \sigma ) - \frac{m}{2} \leq 2 - r $, and the conclusion follows.

\end{proof}

\begin{defn}
 Let $ \{ A_{ i, N } \}_{  i \in I , N \in \mathbb{Z}_+ } $ be an ensemble of random matrices with entries in $ \cA $. We will say that the ensemble has \emph{second order free limit distribution} if it has limit distribution and, for all
$ i_1, i_2, \dots, i_n \in I $, and collection 
$ \{ p_k \}_{ k \in \mathbb{Z}_+ } $ of non-commutative polynomials in $ n $ variables, with the notation 
$ Y_k = p_k ( A_{ i_1, N }, \dots, A_{ i_n, N } ) $, we have that 
\begin{itemize}
\item[(1)] $ \displaystyle  \lim_{ N \lra \infty }
 \kappa_2 \left( \Tr( Y_1 ) ,  \Tr(  Y_2) \right) $  exists and it is finite
\item[(2)] $ \displaystyle  \lim_{ N \lra \infty }
 \kappa_r \left( \Tr( Y_1 ) ,  \dots, \Tr(Y_r ) \right)  = 0 $ for all $ r \geq 3 $.
\end{itemize}
\end{defn}

Ensembles of matrices from $ \coprod_{ n=1}^\infty M_n ( \mathbb{C} ) $ with limit distribution have second order free distribution, since free cumulants with constant entries cancel (see, for example \cite{speichernica}); an immmediate consequence of Theorem \ref{thm:1} is that ensembles of semicircular random matrices also have second order free limit distribution.

\subsection{}\label{section:ast}
 The next notion can be seen as an analogue, in our framework, of the notion of second order free independence from \cite{mingo-nica}.

\begin{defn}\label{def:ast}

Consider $ K \in \mathbb{Z}_+ $ and for each $ k \in K $ let 
$ E_k = \{ A_{ i, N }^{( k ) } \}_{ i\in I, N \in \mathbb{Z}_+ } $ be an ensemble of random matrices that has limit distribution.  We will say that the family $ \{ E_k \}_k $ has \emph{Property  ($ \ast $)} if the following hold true:
\begin{enumerate}
\item[(1)] $ \{ E_k \}_k $ is an asymptotically free family with respect to the map $ \tr \otimes \phi $.
\item[(2)] Suppose that $ \{ P_k \}_{ k \in K } $ are non-commutative polynomials in $ p $ variables and 
that
$ k_1, \dots k_{ s+ t} \in K  $ 
with $ k_j \neq k_{ j + 1 } $ for $ j \in [ s+ t ] \setminus \{ s, s+ t \} $ .
Suppose also that $ \{ A_{ 1, N}^{ ( k ) },  \dots,  A_{ p, N}^{ ( k ) } \}_{N} $
are subensembles of $ E_k $ with limit distribution such that 
\[ 
\lim_{ N \lra \infty }  \tr( P_k ( A_{ 1, N}^{ ( k ) },  \dots,  A_{ p, N}^{ ( k ) } ) )= 0
\]
and denote by $ \alpha_j^{( N )} = P_{ k_j } ( A_{ 1, N}^{ ( k_j ) },  \dots,  A_{ p, N}^{ ( k_j ) } ) $. Then
\[ 
\lim_{ N \lra \infty}
 \kappa_2 ( \Tr( \alpha_1^{( N )} \cdots \alpha_{s}^{ ( N ) } ), 
\Tr( \alpha_{s+1}^{ ( N ) } \cdots \alpha_{s+t}^{ ( N ) }  ) 
= 
\delta_{ s, t} \prod_{ j=1}^s \lim_{ N \lra \infty } 
\tr( \alpha_{j}^{ ( N ) } \alpha_{ s+t+1-j }^{ ( N ) } )
\]
\item[(3)] Suppose that $ r \geq 3 $, $ m \in \mathbb{Z}_{+ } $, that 
$ \{ Q_l \}_{ l=1}^r $  are non-commutative polynomials in $ m $ variables 
and that $ k_1, \dots, k_m \in K $.
Suppose also that $ \{ A_{ 1, N}^{ ( k ) },  \dots,  A_{ p, N}^{ ( k ) } \}_{N} $
are subensembles of $ E_k $ with limit distribution and denote
\[
\beta_{l}^{ ( N )} = \Tr\left( Q_l ( A_{ 1, N}^{ ( k_1 ) },  \dots,  A_{ p, N}^{ ( k_1 ) }, \dots, 
A_{ 1, N}^{ ( k_m ) },  \dots,  A_{ p, N}^{ ( k_m ) } )\right).
\]
Then
\[
\lim_{ N \lra \infty } \kappa_r( \beta_1^{( N ) } , \dots, \beta_r^{( N ) } ) = 0.
\]
\end{enumerate}

\end{defn}
In the next two section, that is Section \ref{section:33} and Section \ref{section:34}, we will prove the following result.
\begin{thm}
Let $ \{ f _k \}_{ k \in \mathbb{Z}_+ } $ be an orthonormal set from $ \cH $ 
and let $ \cE $ be an ensemble of constant matrices with limit distribution.
The family of ensembles  $ \cE $ and $ \{ S_N ( f_k ) \}_ N $ ,  $(  k \in \mathbb{Z}_+ )$,
 has Property $( \ast )$.
\end{thm}

Note first that property (1) from Definition \ref{def:ast} is satisfied, since semicircular matrices are free from matrices with constant coefficients (see, for example, \cite{speichernica}).

\subsection{}\label{section:33}
Let $ s_1, \dots, s_n, j_1, \dots, j_m $ be positive integers and let $ A_1^{( N )}, \dots, A_n^{( N) }$, 
$ B_1^{( N )}, \dots, B_m^{ (N ) } $ be constant matrices that are either centered or identity, such that if 
$ s_k = s_{ k +1} $, respectively if $ j_k = j_{ k+1} $, then $ A_k^{( N )} \neq I $, respectively $ B_k^{( N )} \neq I $.

Let $ k_1, \dots, k_n, l_1, \dots, l_m $ be positive integers and 
\begin{align*}
P_s^{(N )} =& S_N(f_{ i_s})^{ k_s } - \tr( S_N(f_{ i_s})^{ k_s } )I \\
Q_s^{(N )} = &S_N(f_{ j_s})^{l_s } - \tr( S_N(f_{ j_s})^{ l_s } )I. 
\end{align*}
In order to show Property (2) from Definition \ref{def:ast}, it suffices to prove the following Lemma:
\begin{lemma}\label{lemma:35}
With the notations above,
\begin{align}
\kappa_2\left( 
\Tr ( A_n^{(N)} P_n^{( N ) } \cdots  \right. & \left. A_1^{(N)} P_1^{( N ) } ), 
\Tr( Q_1^{( N ) }B_1^{(N)}  \cdots  Q_m^{( N ) } B_m^{(N)})
\right) \label{eq:ast}\\
&\hspace{2cm}= \delta_{ n, m } \prod_{ k =1}^n
 \tr( A_k^{(N )} B_k^{(N )} ) \cdot
\tr( P_k^{(N)} Q_k^{(N ) } ).  \nonumber
\end{align}
\end{lemma}
 We will prove (\ref{eq:ast})  in several steps. First, to simplify the notations, we will omit the index $ N $, with the convention that only matrices of the same dimension are multiplied.

Let us focus first to 
\[
E_{ m, n}= \phi \left ( \Tr( A_n S( f_{ s_n })^{ k_n } \cdots A_1 S( f_{ s_1})^{ k_1 } ) 
\Tr ( S( f_{ j_1} )^{ l_1} B_1 \cdots S( f_{ j_m} )^{ l_m}B_m )
\right).
\]
Let $ N_1 = n + k_1 + \dots + k_n $ , $ N_2 = m + l_1 + \dots + l_m $ and $ M = M_1 + M_2 $, $ M_0 = M -m-n $.
 Denoting by  $ a^{(p )}_{i, j},   b^{(p )}_{i, j} $, respectively $ c^{( p )}_{ i, j } $ the $ (i, j) $-entries of
$ A_p, B_p $, respectively $ S(f_p ) $, and denoting by $ \overrightarrow{i} $ the multiindex $ (i_1, i_2, \dots i_M)$, we have that
\begin{equation}\label{eq:emn}
E_{m, n} = \sum_{ \overrightarrow{i} }\phi( a^{(n)}_{ i_1, i_2} c^{( s_n)}_{ i_2, i_3 } c^{( s_n)}_{ i_3, i_4 }
\cdots 
c^{ (s_1 )}_{ i_{ M_1 }, i_1 } c^{( j_1 )}_{ i_{ M_1 + 1 }, i_{ M_1 + 2 }} \cdots 
c^{( j_m)}_{i_{M}-1, i_{ M } } b^{(m)}_{ i_{ M }, i_{ M_1 + 1 } })
\end{equation}
 From the Free Wick Theorem, the expression above is computed as a sum over all non-crossing pair partitions acting on
the  factors of the type $ c_{ i, j }^{( k )} $, more precisely we can write
\[
E_{m, n} = \sum_{ \pi \in NC_2 ( M _0 ) } \sum_{ \overrightarrow{i} \sim \pi } v(\pi) 
\] 
where we write $ \overrightarrow{i} \sim \pi $ if whenever $ \pi $ is pairing $ c_{ i_s, i_t }^{( k)} $ to $ c_{ i_{u}, i_{ v }}^{( l )}$ we have that $ i_s = i_v $ and $ i_t = i_u $, and we denote $ v(\pi) $ for the expression (depending also on $ A_1, \dots, A_n, B_1, \dots, B_m $) that results by pairing the $ c_{ i, j }^{( k )} $'s  according to $ \pi $.

Denote  $ P_s^{\prime} = S( f_{ i_s } )^{ k_s } $  and $ Q_s^{ \prime} = S( f_{ j_s } )^{ l_s }$. 
Suppose that $ \pi $ pairs two consecutive  entries of the type $ c_{ l, k}^{ r } $  from the same $ P_i^{ \prime}$ or $ Q_j^{ \prime} $. 
Without affecting the  generality, we can suppose, to simplify the notations, that the development of  $ P_i^{ \prime} $
contains the  squence 
$ \cdots  c^{( s_i )}_{ i_{ v-1}, i_v } c^{( s_i )}_{ i_{ v}, i_{v+ 1} } c^{( s_i )}_{ i_{ v+ 1}, i_{v + 2 }} 
 c^{( s_i )}_{ i_{ v+ 2 }, i_{v + 3 }} \cdots $
  and that $ \pi $ pairs 
$ c^{( s_i )}_{ i_{ v}, i_{v+ 1} } $ with  $c^{( s_i )}_{ i_{ v+ 1}, i_{v + 2 }} $. If $ \overrightarrow{i} $ is such that  $ \overrightarrow{i} \sim \pi $, then 
 $ i_v = i_{ v+2 } $, hence, eliminating $ c^{( s_i )}_{ i_{ v}, i_{v+ 1} } c^{( s_i )}_{ i_{ v+ 1}, i_{v + 2 }} $ we will obtain a summand from a development as in equation (\ref{eq:emn}) but $ k_i $ is now replaced by $ k_i - 2 $.  Same argument works if the sequence  $ c^{( s_i )}_{ i_{ v}, i_{v+ 1} } c^{( s_i )}_{ i_{ v+ 1}, i_{v + 2 }}  $ is preceded or succeded by entries of constant matrices.

Denote by $ NC_2^v(M_0) $ the set of all pair partitions acting on $ c_{ i, j }^{( k )} $'s that are pairing 
$ c^{( s_i )}_{ i_{ v}, i_{v+ 1} }$ and $  c^{( s_i )}_{ i_{ v+ 1}, i_{v + 2 }} $, denote by $ \overrightarrow{i}^v $ the set of indices $ \overrightarrow{i} $ without $ i_v $ and by $ NC_2^v(M_0 -2 ) $ the set of pair-partitions actiong on 
$ c_{ i, j }^{( k )} $'s without $\{  c^{( s_i )}_{ i_{ v}, i_{v+ 1} },  c^{( s_i )}_{ i_{ v+ 1}, i_{v + 2 }} \}$.
Since 
$ \phi( c^{( s_i )}_{ i_{ v}, i_{v+ 1} } c^{( s_i )}_{ i_{ v+ 1}, i_{v + 2 }}) =\frac{1}{N}\| f_{ s_i} \|^2 = \frac{1}{N}$,
we have that
\begin{align}
\sum_{ \pi \in NC_2^v(M_0) } \sum_{ \overrightarrow{i} \sim \pi }& v(\pi) 
=
\sum_{ \sigma \in NC_2^v(M_0-2 ) } \sum_{ \overrightarrow{i}^v \sim \sigma }[ v(\sigma ) \cdot 
\sum_{ i_v =1}^N \phi(  c^{( s_i )}_{ i_{ v}, i_{v+ 1} } c^{( s_i )}_{ i_{ v+ 1}, i_{v + 2 }}) ]\label{eq:8}\\
=&
\sum_{ \sigma \in NC_2^v(M_0-2 ) } \sum_{ \overrightarrow{i}^v \sim \sigma } v(\sigma ).\nonumber 
\end{align}

Consider now $ NC_2^{[ t ] } (M_0 ) $ the set of all $ \pi $ as above such that $ P_t^{ \prime} $ is invariant under $ \pi $ 
(that is all $ c^{(s_t)}_{ i, j }$ are paired only among themselves; in particular, $ k_t $ must be even). The  restriction of $ \pi $ to $ P_t^{ \prime} $ is then again a ono-crossing pairing; since any non-crossing partition has at least one interval block, iterating (\ref{eq:8}) we obtain
\begin{equation}\label{eq:9}
\sum_{ \pi \in NC_2^{[ t ] }(M_0) }\sum_{ \overrightarrow{i} \sim \pi } v(\pi) = \tr(P_t ) \cdot
 \sum_{ \sigma\in NC_2^{ [ t ] } (M_0-k_t+2 ) } \sum_{ \overrightarrow{i}^{[ t ] } \sim \sigma } v(\sigma) 
\end{equation}
where $ \overrightarrow{i}^{[ t ] } $ is the multiindex formed by the set of all indeces from $ \overrightarrow{i} $ that are not contained only in the factors from $ P_t^{\prime} $ and where $ NC_2^{ [ t ] } (M_0-k_t+2 ) $ denotes the set of all non-crossing pairings acting on all $ c^{( k )}_{ i, j } $ except the ones in $ P_t ^{ \prime}$.

Let us now go back to the computation of the second order free cumulant
$ \kappa_2\left( 
\Tr ( A_n P_n \cdots   A_1 P_1 ), 
\Tr( Q_1B_1 \cdots  Q_m B_m)
\right) $. 
As seen in the proof of Theorem \ref{thm:1}, it develops (following the Free Wick Theorem) as a sum over pair partitions action on the factors of the type $ c^{( k)}_{ i, j } $ 
and connecting $ P_n \cdots P_1 $ with $ Q_1 \cdots Q_m $.
Note that here the partitions are acting on sets of different lengths, due to  the  presence of  terms of type 
$ \tr( S(f_{ i_s})^{ k_s } )I $ in the expressions of $ P_s $'s (and the analogues for $ Q_s $'s). But, according to (\ref{eq:9}), the factors of the type $ \tr( S(f_{ i_s})^{ k_s } ) $ and the partitions leaving invariant $ P_s $ cancel each other, hence, with the notations from (\ref{eq:emn})
\[
\kappa_2\left( 
\Tr ( A_n P_n \cdots   A_1 P_1 ), 
\Tr( Q_1B_1 \cdots  Q_m B_m)
\right)
=
\sum_{ \pi \in NC_2^{ \sim }(M_0)} \nu( \pi )   
\]
where $ \nu( \pi) = \sum_{ \overrightarrow{i} \sim \pi } v( \pi)  $   and  $ NC_2^{ \sim }(M_0) $ is the set of all non-crossing pairings $ \pi $ acting on the factors of type $ c_{ i, j }^{(k ) } $
such that
\begin{itemize}
\item[(1)]$ \pi $ connects  $ P_n^{ \prime} \cdots P_1^{ \prime} $ with $ Q_1^{ \prime} \cdots Q_m^{ \prime} $;
\item[(2)]no $ P_k^\prime $ or $ Q_k^\prime $ is left invariant by $ \pi $.
\end{itemize}
Suppose $ \pi \in NC_2^{ \sim }(M_0) $ is such that $ \nu ( \pi ) \neq 0 $. first note that, by equation (\ref{eq:8}), we can suppose that $ \pi $ does not connect elements from the same $ P_k^\prime $ or $ Q_k^\prime$.
Second, note that if $ \pi $ connects two $ P_k^\prime$'s, then, from the non-crossing property, it may also connects two consecutive ones, say $ P_t^\prime $ and $ P_{ t+ 1}^\prime $. Using again the fact the $ \pi $ is non-crossing, the last factor of $ P_t^\prime $ must be connected to the first factor of $ P_{t + 1}^\prime $.
 Let $ c^{( s_t )}_{ i_{ v - 1} i_{ v } } $ be the last factor of $ P_t^\prime $.
 Then the development of $ E_{m n } $ contains the sequence 
\[
 \cdots c^{( s_t )}_{ i_{ v - 1} i_{ v } } a^{ ( t )}_{ i_{ v } i_{ v+ 1} } c^{( s_{ t + 1 } )}_{ i_{ v + 1 } i_{ v + 2  } }  \cdots.
\]
Since $ c^{( s_t )}_{ i_{ v - 1} i_{ v } }  $ and $ c^{( s_{ t + 1 } )}_{ i_{ v + 1 } i_{ v + 2  } } $ are connected by $ \pi $,
it follows that $  i_{ v }  =  i_{ v + 1 } $ hence $ v( \pi ) $ contains the factor 
$ \sum_{ i_v = 1}^N a^{( t )}_{i_v  i_v }   = \Tr(A_t ) = 0 $, therefore $ v(\pi) $ cancels.
It follows that $ \pi $ connects only $ P_k^\prime $'s with $ Q_l ^\prime$'s. 

We will show next that each $ P_k^\prime $ can be connected to exactly one $ Q_l^\prime $. 
Suppose that $ P_t^\prime $ is connected to more than one $ Q_l^\prime $. Since $ \pi $ is non-crossing and does not connect two different $ Q_l^\prime $'s, it follows that $ P_t^\prime $ is connected to two consecutive $ Q_l^\prime $'s, say with $ Q_r^\prime $ and $ Q_{ r+ 1}^\prime$. Since we can suppose, by (\ref{eq:8}) that $ \pi $ does not connect 
factors of $ P_t^\prime $ among themselves, it follows that two consecutive factors of $ P_t^\prime $ must be connected to the last fact of $ Q_w^\prime $  and to the first factor of $ Q_{ w +1}^\prime $.
Let $ c^{ s_t }_{ i_{ v- 1 } i_v } , c^{ s_t }_{ i_{ v } i_{ v + 1}} $ be the two consecutive factors of $ P_t^\prime $ connected to the last factor of $ Q_w^\prime $, respectively to the first factor of $ Q_{ w + 1}^\prime$
 If  $ B_w  = I $, then $ f_{ j_w} \neq f_{j_{ w + 1} } $, therefore $ f_{ i_t } $ is orthogonal to at least one of the vectors 
$ f_{ j_w} , f_{j_{ w + 1} } $, hence  $ v( \pi) = 0 $. Therefore we must have $ \tr(B_w ) = 0 $ and the sequence 
$  c^{ ( l_w )}_{ i_{ u-1} i_{ u } } b^{(w)}_{ i_{ u } i_{ u + 1} } c^{ ( l_{w + 1 } )}_{ i_{ u+1} i_{ u + 2  } } $ appearing in the development of $ E_{ mn } $ contains the last, respectively the first factors of $ Q_w^\prime $, $ Q_{ w + 1}^\prime$.
Then $ i_u = i_v = i_{ u + 1} $ and $ v( \pi ) $ will contain the factor 
$ \sum_{ i_v = 1}^N b^{( w )}_{i_v  i_v }   = \Tr(B_w ) = 0 $, therefore $ v( \pi) $ must cancel.

We proved that $ \pi $ must connect each $ P_t^\prime $ to exactly one $ Q_l^\prime $ and  no $ P$'s and $ Q $'s among themselves, particularly that $ n = m $. Moreover, since $  \pi $ is noncrossing and does not connect two different $ P$'s or $ Q $ 's , it follows that $ P_k^\prime $ is connected to $ Q_k^\prime $ for all $ k = 1, \dots, n $. 

We will now finish the proof for equation (\ref{eq:ast}). From the argument above, the relation hold true if $ n \neq m $ and 
if $ n = m $, then 
\[
\kappa_2\left( 
\Tr ( A_m P_m \cdots   A_1 P_1 ), 
\Tr( Q_1B_1 \cdots  Q_m B_m)
\right)
=
\sum_{ \pi \in NC_2^{ \sim }(M_0)} \nu( \pi )   
\]

Fix $ \pi \in NC_2^{ \sim }(M_0) $, and let us denote by $ a^{( k)}_{ j_t, j_{ - t } } $, respectively 
by $ b^{ ( k ) }_{ u_{ - t } u_{ t } } $ the entries of $ A_1, \dots, A_m $, respectively $ B_1, \dots, B_m $
that appear as factors in the corresponding development. We will show that $ j_t = u_t $ and $  j_{ - t } = u_{ - t } $,
therefore, from Lemma \ref{lemma:2},  the factors from $ \nu ( \pi ) $ concerning the constant matrices will be 
$ \prod_{ k =1}^m \Tr(A_k B_k ) $.
Since the summands in $ E_{ mn} $ have trace-type developments, the indeces of $ j_{ t }, j_{ - t }, u_{ t }, u_{ - t } $ 
from above are determined by the indeces of the first and last factors of $ P_{ t }^\prime, P_{ t+1}^\prime $ and 
$ Q_{ t }^\prime, Q_{ t + 1}^\prime $. From the argument in equation (\ref{eq:8}), we can suppose, in what concerns indeces, that $ \pi $ does not connect factors from the same $ P_k^\prime $ or $ Q_k^\prime$, hence it may suppose that each factor from $ P_k^\prime$ is connected to a factor from $ Q_k^\prime$. Since  $ \pi $  is non-crossing, the last, respectively first, element from $ P_k^\prime$ must be connected to the first, respectively last, factor from $ Q_k^\prime$, and $ \overrightarrow{ i } \sim \pi $ gives the result.

Fix not $ t \in \{ 1, \dots, m \}$. From the argument above, if $ \pi \in NC_2^{ \sim }(M_0) $ such that 
$ \nu( \pi) \neq 0 $, then the set $ C_t =\{ c^{( k)}_{ i, j } : c^{( k)}_{ i, j } 
\text{  is a factor in $ P_t^\prime$ or in $ Q_t^\prime $ } \} $ is invarinat under $ \pi $. Hence, denoting by $ \pi_t $ the restriction of $ \pi $ to $ C_t $, the Free Wick Theorem implies that $ v ( \pi_t )$, respectively $ \nu( \pi_t) $ factors 
in $ v(\pi) $, respectively in $ \nu( \pi) $. Let us also denote by $ \overrightarrow{ i} ( t ) $ the set of indeces from 
$ \overrightarrow{ i } $  that appear as lower indeces for elements of $ C_t $.  Let us write
\begin{align*}
P_t^\prime = &  c^{ ( s_t )}_{ i_{ v } i_{ v + 1 } } c^{ ( s_t )}_{ i_{ v+1 } i_{ v + 2} }
\cdots  c^{ ( s_t )}_{ i_{ v+ k_t -1 } i_{ v + k_t } }
\\
Q_t^\prime = & 
c^{ ( j_t )}_{ i_{ w } i_{ w + 1 } } c^{ ( j_t )}_{ i_{ w+1 } i_{ w + 2} }
\cdots  c^{ ( j_t )}_{ i_{ w+ l_t -1 } i_{ w + l_t } }.
\end{align*}
Then, the previous argument gives that $ i_v = i_{ w + l_t } $ and $ i_{ v + k_t } = i_{ w } $. Since $ \pi_t $ connects
$ P_t^\prime $ to $ Q_t^\prime $, equation (\ref{eq:emn}) gives that 
\begin{align*}
\nu( \pi_t ) 
& = \phi( \Tr( P_t^\prime Q_t^\prime ) - \phi( \Tr( P_t^\prime )) \phi( \Tr ( Q_t^\prime ))\\
 & = \phi( \Tr(P_t Q_t )
\end{align*}
Remark now that the indices $ i_w $ and  respectively $ i_v $ are counted both in $ \Tr( P_t Q_t ) $ and in $ \Tr(A_{t}B_{t})$, respectively $ \Tr( A_{ t + 1} B_{ t+ 1} ) $, henceforth
\[
\sum_{ \pi\in NC_2^{ \sim} ( M_0) } \nu(\pi) = \prod_{ t =1}^m \frac{1}{ N^2 } \Tr(A_t B_t ) \Tr( P_tQ_t ) =
\prod_{ t =1}^m \tr(A_t B_t ) \tr( P_tQ_t ),
\]
hence  the proof of (\ref{eq:ast}) is concluded.

\begin{remark}
Lemma \ref{lemma:35} can be seen as a free analogue of Theorem 5.3 from \cite{mingo-speicher}; yet, 
the  results are different in nature, Theorem 5.3 from \cite{mingo-speicher} is an asymptotical result, more in the spirit 
of part $(2)$  from Property $(\ast )$.
\end{remark}

\subsection{Vanishing of higher order free cumulants}\label{section:34} ${}$\\
Suppose that $ \{ f_i \}_{ i \in \mathbb{Z}_+ } $ is an orthonormal system in $ \cH $,
 let $ i_1, \dots, i_m \in \mathbb{Z}_+ $  and let $ S_N( f_i ) $ be as defined in Section \ref{section:31}.

Let $ l_1, \dots, l_r > 0 $ and put $ M ( 0 ) = 0 $, $ M( k ) = M ( k - 1 ) + l_k $, for $ k \in \{ 1, \dots, r - 1 \} $, 
and $ M = M(r ) $.

Suppose that 
$ \{ A_1^{( N )}, \dots, A_M^{ (N ) } \}_N $ is an ensemble of constant matrices with limit distribution (some of them may be identity matrices) and, for $ k = 1, \dots, r $, define
\[
Y_k^{ ( N ) } = 
S_N (f_{ i_{ M( k - 1 ) } + 1 } ) A^{( N ) }_{ { M( k - 1 ) } + 1 } \cdots 
S_N ( f_{ i_{ M( k )} } ) A^{( N ) }_{ M ( k ) }. 
\]

\begin{thm}\label{thm:34}
 With the notations from above, if $ r \geq 3 $, we have that
\[
\lim_{ N \lra \infty }\kappa_r ( \Tr(Y_1^{ ( N )} ), \Tr ( Y_2^{ ( N ) } ), \dots, \Tr ( Y_r^{ ( N ) } ).
\]
\end{thm}

\begin{proof}

As before, we will omit the index $ N $, with the convention that only matrices of the same size are multiplied. Also,
we will denote by $ a^{( k )}_{ i, j } $, respectively $ c^{( k )}_{ i, j } $ the $ ( i, j ) $  entry of $ A_k $, 
respectively $ S_N( f_k ) $.

Let $ \gamma $ be the  permutation with $ r $ cycles 
$ \left( M( k -1) + 1, M(  k - 1 ) + 2, \dots, M( k ) \right)$
 and $ \widehat{ \gamma } $ be the permutation with $ r $ cycles 
$ \left( 2 M( k -1) + 1, 2 M(  k - 1 ) + 2, \dots, 2  M( k ) \right) $
for $ 1 \leq k \leq r $ . 
Denote  $ \overrightarrow{ i } = ( i_1, \dots, i_m ) \in \mathbb{Z}_+^M $ and 
by $ \overrightarrow{j} = ( j_1, \dots, j_m ) \in [ M ]^N $. 
Then
\[
 \Tr( Y_1)\cdots \Tr( Y_r )  = \sum_{ \overrightarrow{j} }
\left( 
 c^{( i_1)}_{ j_1 j_2} \cdots c^{( i_M)}_{ j_{ 2 M -1}j_{ 2 M }}
\right)
 \cdot 
 \left(   
a^{( 1)}_{ j_2 j_{ \widehat{\gamma}(2) }} \cdots 
a^{( M)}_{ j_{ 2 M } j_{ \widehat{\gamma}(2 M ) }} 
\right).
\]

Let $ \pi \in NC_2 ( M ) $. We will put 
$ \delta_{ \overrightarrow{i}}^\pi $
 to be 1 if $ \pi(l) = k $ implies $ i_l = i_k $ and 0 otherwise.
The pairing $ \pi $ induces a pairing $ \wpi\in NC_2(2M) $ as follows:  
if $ \pi ( l ) = k $, then put $ \wpi ( 2l - 1 ) = 2 k $ and $ \wpi( 2 l ) = 2 k -1 $.

Since $ \phi ( c^{( a )}_{ i j } c^{( b )}_{ k l } ) = \delta_{a b } \delta_{ j k } \delta_{ i l } $, the  
Free Wick Theorem implies that
\[
\phi\left( \Tr( Y_1)\cdots \Tr( Y_r ) \right)
=\sum_{ \pi\in NC_2(M) } 
\delta_{ \overrightarrow{i}}^\pi
N^{ -\frac{M}{2}} \sum_{ \overrightarrow{j} = \overrightarrow{j}\circ \wpi }
a^{( 1)}_{ j_2 j_{ \widehat{\gamma}(2) }} \cdots 
a^{( M)}_{ j_{ 2 M } j_{ \widehat{\gamma}(2 M ) }} 
.
\]
Denoting by
 $ \displaystyle
\mu(\pi)= \delta_{ \overrightarrow{i}}^\pi
N^{ -\frac{M}{2}} \sum_{ \overrightarrow{j} = \overrightarrow{j}\circ \wpi }
  a^{( 1)}_{ j_2 j_{ \widehat{\gamma}(2) }} \cdots 
a^{( M)}_{ j_{ 2 M } j_{ \widehat{\gamma}(2 M ) }} 
$, 
an inductive argument on $ r $ gives that
\[
\kappa_r ( \Tr( Y_1), \dots, \Tr(Y_r) ) = \sum_{ \substack{ \pi \in  NC_2(M )\\ \pi \vee \gamma = 1_M }}
\mu( \pi ).
\]
We will show Theorem \ref{thm:34} by proving that if $ \pi \in NC_2(M ) $ is such that $ \pi \vee \gamma = 1_M $, then
$ \mu( \pi) = O( N^{ -1} )$.

Fix $ \pi $ as above. Appling Lemma \ref{lemma:2} to $ \wpi \in NC_2(2M) $ and $ \sigma \in S_{  2M } $  
given by $ \sigma( 2 k ) = \widehat{ \gamma }( 2 k - 1 ) $ and $ \sigma ( 2 k + 1 ) = 2 k + 2 $, we have that
there exist some $ \tau \in S_M $ such that
\[
\sum_{ \overrightarrow{j} = \overrightarrow{j}\circ \wpi }
  a^{( 1)}_{ j_2 j_{ \widehat{\gamma}(2) }} \cdots 
a^{( M)}_{ j_{ 2 M } j_{ \widehat{\gamma}(2 M ) }}
=
\Tr_{ \tau } ( A_1, \dots, A_M ).
\]
Since lower indeces of factor of type $ a^{(k)}_{ i  j } $ are from th same block of $ \widehat{\gamma} $,
we have that $( k ) $ is a singleton of $ \tau $ only if $ \pi( k) = k + 1 $ and both $ k $ and $ k + 1 $ 
are from the same block of $ \gamma $. As seen in Section \ref{section:33}, in this situation we can simply remove 
$ S_N(f_{ i_k } ) A^{ ( N )}_k S_N ( f_{ i_{ k + 1}}) $
from the product without affecting the order of magnitude of the product.
 Henceforth, we can supposse that $ \tau $ does not have singletons and that $ \pi $ does not connect elements from the same block of $ \gamma $.

Since the ensemble $ \{ A_1^{( N )}, \dots, A_M^{ ( N ) } \}_N $ has limit distribution, we have that 
\[
\Tr_\tau (A_1, \dots, A_M) = O(N^{ \# (\tau) } )
\]
so Theorem \ref{thm:34} is proved if we show that, under the asumptions above, $ \tau $ has a cycle with at least 3 elements.

Let $ B_1, \dots, B_r $ be the blocks of $ \gamma $. Since $ \pi \vee \gamma = 1_M $, there is at least one block of 
$ \gamma $ connected by $ \pi $ with more than one other block. 
Suppose that $ B_k =( M( K -1) +1, \dots, M( k ) ) $ is such a block and that
 $ \pi( M( k ) ) \in B_l $.
 If $ \pi( M ( k -1)+ 1 ) \in B_l $, since $ \pi $ is non-crossing, we would have that $ \pi( B_k ) \subset B_l $, therefore
$ \pi( M( k -1 ) + 1 ) \not \in B_l $.

Let  $ v = -1 + \inf \{ t: \ t \in B_k, \pi( t) \in B_l \}  $, let $ \pi( v + 1 ) = w \in B_l $ and take $ u, s $ such that
$ \pi(v) = u \in B_s \neq B_l $. If $ \overrightarrow{j} = \overrightarrow{j}\circ \wpi $, then
$ j_{ 2v} = j_{ 2u + 1} $ and $ j_{ 2v+ 1 } = j_{ 2 w } $, 
which implies that $ A_u A_v A_w $ are in the same cycle of $ \tau $, hence the conclusion.

\end{proof}

\section{Random matrices with Boolean independent Bernoulli-distributed entries}\label{section:boolean}

\subsection{}
As in Section \ref{section:free}, we will consider $ H $ to be a real Hilbert space, $ \cH = H \otimes \mathbb{C} $.
Let $ \{ B_N( f ) \}_{ N \in \mathbb{Z}_+, f \in \cH } $ be an ensemble of random
matrices such that
$ B_N ( f ) = \left[   b_{ i, j }^{( N )} ( f ) \right]_{ i, j =1}^N $ with 
$ b_{ i, j } ( f) \in \cA $  such that
\begin{itemize}
\item[(i)]$\displaystyle  \phi \left( b^{( N )}_{ i, j } ( f ) b^{(N)}_{ k, l } ( g ) \right) 
=\frac{1}{N} \langle f, g \rangle \delta_{ i, k } \delta_{ j, l } $
\item[(ii)]$ B_N( f)^\ast = B_N(f) $
\item[(iii)]$ \{ \Re b^{( N )}_{i, j}(f), \Im  b^{(N)}_{i, j }( f) \}_{ 1 \leq i, j \leq N }$  form a Boolean independent family of Bernoulli distributed random variables of mean 0.
 \end{itemize}
\begin{remark}${}$
\begin{itemize}
\item[(1)] $ B_N ( f) $ is Bernoulli distributed of variance $ \| f \|^2 $ $($with respect to the functional $ \tr \otimes \phi )$.
\item[(2)]If $ \{ f_i\}_{ i \in \mathbb{Z}_+ } $ is an orthonormal family in $ \cH $, then $ \{ B_N( f_i ) \}_{ i } $ form a Boolean independent family.
\end{itemize}
\end{remark}

\begin{proof}

Consider $ m_1, \dots, m_p \in \mathbb{Z}_+ $, $ m = m_1+ \dots + m_p $ 
and the multiindex
$ \overrightarrow{ i } = ( i_1, i_2, \dots, i_m)\in [N ]^m $.

To simplify the writing we will omit the upper-index $ ( N ) $, with the convention that only matrices of the  same size 
are multiplied.

For part (1), note that
\[
\tr\otimes \phi( B_N (f)^m ) = \frac{1}{N} \sum_{ \overrightarrow{i} }
\phi\left( b_{ i_1 i_2} b_{ i_2 i_3 } \cdots b_{i_m i_1} \right).
\]
If $ m $ is odd, Proposition \ref{wick:boolean} gives that all summands cancel. If $ m $ is even, Proposition \ref{wick:boolean} gives that
\[
\tr\otimes \phi( B_N (f)^m ) =
\sum_{\overrightarrow{i}} N^{ -\frac{m}{2}-1}\| f\|^m \delta_{ i_1, i_3}\delta_{ i_3 i_5}\dots
\delta_{i_{ m -3} i_{ m-1} } \delta_{ i_{ m-1}i_1}= \| f\|^m.
\]

 For part (2), it suffices to prove that, if $ j_k \neq j_{ k+ 1}$, then

\begin{align*}
 \tr\otimes \phi\left( 
B_N( f_{ j_1})^{ m_1} \cdots\right.
 & \left. B_N( f_{ j_p })^{ m_p }
\right)\\
&= 
[ \tr\otimes \phi\left( 
B_N( f_{ j_1})^{ m_1} \right)] \cdot [  \tr\otimes \phi\left( 
B_N( f_{ j_2})^{ m_2} \cdots B_N( f_{ j_p })^{ m_p } 
\right)].
\end{align*}
On the other hand, 
\begin{align*}
 \tr\otimes \phi & \left( 
B_N( f_{ j_1})^{ m_1} 
 \cdots  B_N( f_{ j_p })^{ m_p }
\right)\\
&=\sum_{ \overrightarrow{i}} \frac{1}{N}
\phi \left(
b_{i_1 i_2 }( f_{ j_1 } ) b_{i_2 i_3 }( f_{ j_1 } ) 
\cdots 
b_{i_{ m_1} i_ { m_1 + 1}}( f_{ j_1 } ) b_{i_{ m_1+1} i_ { m_1 + 1}}( f_{ j_2 } ) 
\cdots 
b_{i_m i_1}( f_{ j_p } )
\right)
\end{align*}
If $ m_1 $ is odd, then, from part (1), $ \tr\otimes \phi ( B_N( f_{j_1})^{ m_1}) = 0$;
also, applying Proposition \ref{wick:boolean} to the equation above, we obtain
\begin{align*}
\phi \left(
b_{i_1 i_2 }( f_{ j_1 } )\right) &
\left. b_{i_2 i_3 }( f_{ j_1 } ) 
\cdots 
b_{i_{ m_1} i_ { m_1 + 1}}( f_{ j_1 } ) b_{i_{ m_1+1} i_ { m_1 + 1}}( f_{ j_2 } ) 
\cdots 
b_{i_m i_1}( f_{ j_p } )
\right)\\
&=
\phi( b_{i_{ m_1} i_ { m_1 + 1}}( f_{ j_1 } ) b_{i_{ m_1+1} i_ { m_1 + 1}}( f_{ j_2 } ) )
\cdot
\prod_{ k=1}^{ m_1 - 3} 
\phi\left( b_{ i_k i_{ k+1} } (f_{ j_1}) b_{ i_{k+1} i_{ k+2} } (f_{ j_1})\right) \\
& \hspace{2cm} \cdot 
\phi\left(
b_{ i_{ m_1 + 1 } i_{ m_1 + 2} }( f_{ j_2 } )
\cdots
b_{i_m i_1}( f_{ j_p } )
\right).
\end{align*}
Since $ f_{ j_1} \perp f_{j_2} $ , we have that 
\[
 \phi\left( b_{ i_k i_{ k+1} } (f_{ j_1}) b_{ i_{k+1} i_{ k+2} } (f_{ j_1})\right) = 
\langle f_{ j_1}, f_{j_2} \rangle \delta_{ i_{ k } i_{ k+ 2} } = 0
\]
hence the equality holds true.

If $ m_1 $ is even, Proposition \ref{wick:boolean} gives
\begin{align*}
\phi \left(
b_{i_1 i_2 }( f_{ j_1 } )\right) &
\left. b_{i_2 i_3 }( f_{ j_1 } ) 
\cdots 
b_{i_{ m_1} i_ { m_1 + 1}}( f_{ j_1 } ) b_{i_{ m_1+1} i_ { m_1 + 1}}( f_{ j_2 } ) 
\cdots 
b_{i_m i_1}( f_{ j_p } )
\right)\\
&=
[\prod_{ k=1}^{ m_1 - 3} 
\phi\left( b_{ i_k i_{ k+1} } (f_{ j_1}) b_{ i_{k+1} i_{ k+2} } (f_{ j_1})\right) ]
\cdot 
\phi\left( b_{ i_{ m_1 + 1 } i_{ m_1 + 2} }( f_{ j_2 } )
\cdots
b_{i_m i_1}( f_{ j_p } )
\right)
\end{align*}

Since
 $ \phi\left( b_{ i_k i_{ k+1} } (f_{ j_1}) b_{ i_{k+1} i_{ k+2} } (f_{ j_1})\right) 
= \| f_{ j_1} \|^2 \delta_{ i_k i_{ k+2}} 
 $,
the right-hand side of the  equation above cancels unless $ i_1 = i_3 =\dots = i_{ m_1+ 1} $,
hence, denoting 
$ \overrightarrow{i}(m_1+ 1) = ( i_{ m_1 + 1}, i_{m_1 +2}, \dots, i_{ m } ) $, we obtain
\begin{align*}
 \tr\otimes \phi & \left( 
B_N( f_{ j_1})^{ m_1} 
 \cdots  B_N( f_{ j_p })^{ m_p }
\right)\\
&=\sum_{ \overrightarrow{i}} \frac{1}{N}
\phi \left(
b_{i_1 i_2 }( f_{ j_1 } ) b_{i_2 i_3 }( f_{ j_1 } ) 
\cdots 
b_{i_{ m_1} i_ { m_1 + 1}}( f_{ j_1 } ) b_{i_{ m_1+1} i_ { m_1 + 1}}( f_{ j_2 } ) 
\cdots 
b_{i_m i_1}( f_{ j_p } )
\right)\\
&=
N^{ -\frac{m_1}{2}-1} \cdot N^{ \frac{m_1}{2} }
\| f_{ j_1}\|^{ m_1} \cdot 
\sum_{ \overrightarrow{i}(m_1+ 1 ) }
\phi
\left(b_{ i_{ m_1 + 1 } i_{ m_1 + 2} }( f_{ j_2 } )
\cdots
b_{i_m i_{ m_1 + 1}}(f_{ j_p } )
\right)\\
&=
[ \tr\otimes \phi \left( B_N( f_{ j_1})^{ m_1} \right) ] \cdot 
[  \tr\otimes \phi\left( 
B_N( f_{ j_2})^{ m_2} \cdots B_N( f_{ j_p })^{ m_p } 
\right)].
\end{align*}

\end{proof}
 Since the results in the next two sections are  not asymptotic in nature, we will omit the index $ N $ from the notation
$ B_N ( f) $, with the convention that all matrices involved here are of size $ N $ for some $ N \geq 2 $.

\subsection{Monotone independence and matrices with Bernoulli distributed Boolean independent entries}${}$

As presented in \cite{sbg}, the relation of Boolean independence is not unital, that is $ \mathbb{C}$
is not Boolean independent from any algebra. In this section we will show that, when tensoring with matrices, the relation of 
\emph{monotone independence} appears connecting Boolean independence to the algebra $ M_N(\mathbb{C}) $.

\begin{lemma}\label{lemma:42}
Let $ A_1, \dots, A_m$ be a set of matrices from $ M_N( \mathbb{C} )$, $ \{ f_k \}_{ k \in \mathbb{Z}_+ }$ be 
a set of vectors from $ \cH $ and $ \overrightarrow{i} = ( i_1, i_2, \dots, i_{ M + 1} )\in [ N ]^{ M + 1 } $.
Then
\begin{align*}
(\tr\otimes \phi )
& ( 
\prod_{ k = 1 }^{ m } B ( f_{ 2k - 1 } ) A_ k B ( f_ { 2 k } ) 
)
= \prod_{ k = 1 }^{ m } \tr ( A_ k ) \langle f_{ 2k - 1 } , f_{ 2 k } \rangle\\
& =
\phi
(
\sum_{ \overrightarrow{i} } \prod_{ k = 1 }^{ m } b_{ i_{ 1 } i_{  k + 1  } }( f_{ 2 k - 1 } )
 a^{ ( k ) }_{ i_{ k + 1  } i_{ k + 1   }  } 
 b_{ i_{  k + 1 } i_{ 1 } }( f_{ 2 k } )
)
\end{align*}
\end{lemma}

\begin{proof}

Denote 
$ \displaystyle
 X = [ x_{ i, j} ]_{ i, j =1}^N =  \prod_{ k = 2  }^{ m } B ( f_{ 2k - 1 } ) A_ k B ( f_ { 2 k } ). 
$
Then
\[
(\tr\otimes \phi )
 ( 
 B ( f_{ 1 } ) A_1 B ( f_ { 2  }) X  
)
= \frac{1}{N} 
\sum_{ i_1 i_2 i_3 = 1}^{ N }
\phi
\left(
b_{i_1 i_ 2 }( f_1 ) a^{( 1)}_{ i_2 i_3 } b_{ i_3 i_ 4 }( f_2 ) x_{ i_4 i_ 1 } 
\right).
\]
Since for all $ 1 \leq i, j \leq N $, $ x_{ i j } $ is in the unital algebra generated by $ b_{ k, l }( f_p ) $,
Proposition \ref{wick:boolean} gives that 
\begin{align*}
\sum_{ i_1 i_2 i_3 = 1}^{ N }
\phi
\left(
b_{i_1 i_ 2 }( f_1 ) a^{( 1)}_{ i_2 i_3 } b_{ i_3 i_ 4 }( f_2 ) x_{ i_4 i_ 1 } 
\right)
&=
\sum_{ i_1 i_2 i_3 = 1}^{ N }
\phi(
b_{i_1 i_ 2 }( f_1 )  b_{ i_3 i_ 4 }( f_2 ))
a^{( 1)}_{ i_2 i_3 }
\phi( x_{ i_4 i_ 1 } )\\
=&
\sum_{ i_1 i_2 i_3 = 1}^{ N }
\frac{1}{N}\langle f_1, f_2 \rangle \delta_{i_1 i_4}\delta_{ i_2 i_3 }
a^{( 1)}_{ i_2 i_3 }
\phi( x_{ i_4 i_ 1 } )\\
=&
\sum_{ i_1 i_2 }^N
\frac{1}{N}\langle f_1, f_2 \rangle a^{( 1)}_{ i_2 i_2 }\phi( x_{ i_1 i_ 1 } )\\
=&
\tr( A_1 ) \langle f_1, f_2 \rangle \Tr( X ).
\end{align*}

\end{proof}

\begin{thm}\label{thm:monotone}
 Let $ \cB $ be the (non-unital) algebra generated by $ \{ B( f ) A B( g ) : A\in M_N( \mathbb{C} ), f, g \in \cH \} $.
Then $ \cB $ is monotone independent from $ M_N(\mathbb{C})$ with  respect to the functional 
$ \tr \otimes \phi $.
\end{thm}

\begin{proof}
 It suffices to show that, if 
$A, D \in M_N( \mathbb{C} ) $, $ f_1, f_2 \in \cH $ and 
$ X $  is in the algebra generated by 
$ M_N ( \mathbb{C} ) $ and $ \{ B( f ) : \ f \in \cH \} $,
then
\[
\tr \otimes \phi ( A B( f_1) D B( f_2 ) X )  = \tr(A)\langle f_1, f_2 \rangle [ \tr  \otimes \phi ( AX ) ]
\]
and the conclusion follows appling Lemma \ref{lemma:42}.

Denoting $ \overrightarrow{i} = (  i_1, i_2, \dots, i_5 ) $, we have that
\[ 
\tr \otimes \phi ( A B( f_1) D B( f_2 ) X ) = \frac{1}{N}
\sum_{ \overrightarrow{i} } 
\phi(
a_{  i_1  i_2 } b_{  i_2  i_3 }( f_1 ) d_{ i_3  i_4 } b_{ i_4  i_5  }( f_2 )  x_{  i_5  i_1  }
).
\]
But, since all  $ x_{ i, j } $ are in the unital algebra generated by 
$ \{  b_{ k l }( f ) : \ f \in \cH \} $,
Proposition \ref{wick:boolean} gives:
\begin{align*}
\sum_{ \overrightarrow{i} } 
\phi(
a_{  i_1  i_2 } b_{  i_2  i_3 }( f_1 ) d_{ i_3  i_4 } b_{ i_4  i_5  }( f_2 )  x_{  i_5  i_1  }
)
&=
\sum_{ \overrightarrow{i} } 
\phi( b_{  i_2  i_3 }( f_1 )b_{ i_4  i_5  }( f_2 ) ) d_{ i_3 i_4 }\phi( a_{ i_1 i_2 } x_{  i_5  i_1  } )\\
& \hspace{-3cm} =
\sum_{ \overrightarrow{i} } \frac{1}{N}
\langle f_1, f_ 2 \rangle \delta_{ i_3 i_4 } \delta_{ i_2  i_5 } d_{ i_3 i_4 }
\phi( a_{ i_1 i_2 } x_{  i_5  i_1  } )\\
& \hspace{-3cm} =
\sum_{ i_1, i_2, i_3 = 1}^N 
\langle  f_1, f_2 \rangle  \frac{1}{N} d_{ i_3 i_3 } \phi( a_{ i_1 i_2 } x_{ i_2 i_1 } )
= \tr(A) \langle f_1, f_2 \rangle \Tr( A X ).
\end{align*}

\end{proof}

\subsection{} Let $ f_1, \dots, f_M \in \cH $,
 let $ l_1, \dots, l_r > 0 $ and put $ M ( 0 ) = 0 $, $ M( k ) = M ( k - 1 ) + l_k $, for $ k \in \{ 1, \dots, r - 1 \} $, 
and $ M = M( r ) $.

Suppose that 
$  A_1, \dots, A_M \in M_N( \mathbb{C} ) $ and, for $ k = 1, \dots, r $, define
\[
Y_k = 
B (f_{  M( k - 1 )  + 1 } ) A_{  M( k - 1 )  + 1 } \cdots 
B ( f_{  M( k ) } ) A_{ M ( k ) }. 
\]

\begin{thm}
With the notations above, 
\[
\mathfrak{b}_r (\Tr( Y_1), \dots, \Tr( Y_r) ) = O( N^{ 2 - r } ).
\]
\end{thm}

\begin{proof}
 Let $ \overrightarrow{i} = ( i_1, \dots, i_{ 2 M }) \in [ N ]^{ 2 M } $ and $ \gamma \in S_{ 2 M } $ be the permutation
with $ r $ cycles $ ( 2 M( k -1) + 1, 2 M( k - 1 ) + 2 , \dots, 2 M ( k ) ) $, for $ 1\leq k \leq r $.
Proposition \ref{wick:boolean} gives
\begin{align*}
\phi( \Tr(Y_1) \cdots \Tr(Y_r ) ) & = \sum_{ \overrightarrow{i} } \phi(
 \prod_{ k = 1}^{ M }
b_{ i_{ 2 k - 1 } i_{ 2 k } } ( f_ k ) a^{( k )}_{ i_{ 2 k } i_{ \gamma ( 2 k ) } }
)\\
& \hspace{-1.5cm} =
\sum_{ \pi \in I_2 ( M ) } \{ 
\sum_{ \overrightarrow{i} }
\{
[ \prod_{ ( k, l ) \in \pi } \phi( b_{ i_{ 2 k - 1 } i_{ 2 k } } ( f_ k ) b_{ i_{ 2 l - 1 } i_{ 2 l } } ( f_ l )  )  ]
\cdot
[ \prod_{ s = 1}^M a^{( s ) }_{ i_{ 2 s }, i_{ \gamma  ( 2 s ) } } ]
\}
\}.
\end{align*}
Induction on $  r $ gives  that $ \mathfrak{ b }_r $ is a sum as above but over $ \pi \in I_2 ( M ) $ such that
$ \pi \vee \gamma^\prime = 1_m $, for  $ \gamma^\prime \in P( M)$ the partition with $ r $ blocks of type
 $( M( k -1) + 1, M( k - 1 ) + 2 , \dots, M( k ) ) $, where $ 1 \leq k \leq r $. In particular, since $ I_2 (M)$ has at most one element, $ \mathfrak{ b }_ r( \Tr(Y_1) \cdots \Tr(Y_r ) ) = 0 $ unless such a pairing exists, that is if $ r = $, then $ l_1 $ is odd, and if $ r \geq 2 $, then  $ l_1 $ and $ l_r $ are odd and 
$ l_2, \dots, l_{ r -1} $ are even.

Let us suppose that there exists $ \pi \in I_2( M )$ satisfying the conditions above. In this case, from the  expansion
(3), we also have that 
\[
 \mathfrak{b}_r ( \Tr(Y_1) \cdots \Tr(Y_r ) ) = \phi ( \Tr(Y_1) \cdots \Tr(Y_r ) ).
\]
Suppose now that  $ l_1 \geq 2 $ and let $ \alpha  = \Tr( Y_2 ) \cdots \Tr( Y_r ) $ (if $ r = 1$, we put $ \alpha = I_N $)
and $ Y_1^\prime = Y \cdot A_2 $, where 
$Y =  B( f_3 ) A_3 \cdots B( f_{ M ( 1 ) }) A_{ M ( 1 ) }  = [ y_{ i, j }]_{ i, j =1}^N$ (if $ l_1 = 2 $, we put $ Y =A_{M(1)}$). 

Then, for $ \overrightarrow{j} = ( j_1, \dots, j_5 ) \in [ N ]^5 $, Proposition \ref{wick:boolean} gives
\begin{align*}
\phi ( \Tr(Y_1) \cdots \Tr(Y_r ) )
&= \sum_{ \overrightarrow{j} }
\phi(
b_{ j_1 j_2 } ( f_ 1 ) a^{ ( 1 ) }_{ j_2 j_3} b_{ j_3 j_4 }(f_2) a^{ ( 2 )}_{ j_4 j_5 } y_{ j_5 j_1 } \cdot \alpha 
)\\
&= \sum_{ \overrightarrow{j} } 
\phi ( b_{ j_1 j_2 } ( f_ 1 )b_{ j_3 j_4 }(f_2)  ) a^{ ( 1 ) }_{ j_2 j_3}
 \phi( y_{ j_5 j_1 }a^{ ( 2 )}_{ j_4 j_5 } \cdot \alpha ).
\end{align*}
Since 
$ \displaystyle \phi ( b_{ j_1 j_2 } ( f_ 1 )b_{ j_3 j_4 }(f_2)  ) 
=\frac{1}{N} \langle f_1, f_2 \rangle \delta_{j_1, j_4 } \delta_{j_2 j_3} $, it follows that
\begin{align}
\phi ( \Tr(Y_1) \cdot\alpha )
&=\sum_{ i, j, k = 1}^N  \langle f_1, f_2 \rangle \frac{1}{N} a^{ (1)}_{ i, i } 
\phi( y_{ j k} a^{( 2 )}_{ k j } \cdot \alpha )\label{eq:10}\\
&  = 
\langle  f_1, f_2 \rangle  \tr(A_1) \phi( \Tr ( Y_1^\prime ) \cdot \alpha ) 
= O(N^0) \phi( \Tr ( Y_1^\prime ) \cdot \alpha ).\nonumber
\end{align}

If $ r = 1$, iterating equation (\ref{eq:10}), we obtain
\begin{equation}
\phi\left( \Tr( B ( f_1 ) A_1 \cdots B( f_{ 2 m } ) A_{ 2 m } ) \right)
=[ \prod_{ k =1}^m 
\langle f_{ 2k -1}, f_{ 2k } \rangle \tr(A_{ 2k -1 } ) ] 
\cdot
\phi( \Tr( \prod_{ j =1}^m A_{ 2 j } ) )
\end{equation}
which implies the theorem for $ r= 1 $.

For $ r \geq 2 $, a similar argument to (\ref{eq:10}) aplied to the last two factors of the type
$ b_{i, j}( f_k ) $ of $ Y_r $ (if $ l_r > 1 $) 
and to the second and third such factors from  $ Y_k $ 
( if $ l_k > 2 $ for $ 2 \leq k \leq r - 1 $)
 gives that  it suffices to prove the theorem for $ l_1 = l_r = 1 $ and 
$ l_2 = \dots = l_{ r - 1}=2$.

Let $ Y_1 = B( f_1) A_1 $ and let $ Y_2 = B(f_2) X $  with $ X = A_2 $ if $ r = 2 $, respectively $ X = A_2 B(f_3) A_3 $ if 
$ r \geq 3 $. put $ \alpha = I_N $ if $ r =2 $, respectively $ \alpha = \Tr(Y_3)\cdots \Tr(Y_r) $ if $ r \geq 3 $.
With this notations, we have that
\begin{align*}
\phi(\Tr(Y_1)\Tr(Y_2) \cdot \alpha) 
&=
\sum_{i, j, k, l =1 }^N 
\phi( b_{ i, j} ( f_1) a^{ ( 1) }_{j, i } b_{ k, l }( f_2) x_{l , k } \cdot \alpha)\\
&=
\sum_{i, j, k, l =1 }^N
\phi( b_{ i, j} ( f_1)b_{ k, l }( f_2) ) \phi ( x_{l , k } a^{ ( 1) }_{j, i }\alpha)
\end{align*}
Since 
$\displaystyle \phi( b_{ i, j} ( f_1)b_{ k, l }( f_2) ) 
 = \frac{1}{N}\langle f_1, f_2 \rangle \delta_{ i, l } \delta_{ j, k } $,
 it follows that
\begin{align}
\phi(\Tr(Y_1)\Tr(Y_2) \cdot \alpha) 
&=\frac{\langle f_1, f_2 \rangle } { N } \sum_{ i, j =1}^N 
\phi( x_{ i, j } a^{ (1) }_{ j, i} \cdot \alpha ) \label{eq:12}\\
&= \frac{\langle f_1, f_2 \rangle } { N } \sum_{ i, j =1}^N 
\phi( \Tr(XA_1) \cdot \alpha ).\nonumber
\end{align}
If $ r = 2 $, equation (\ref{eq:12}) implies that $ \phi(\Tr(Y_1)\Tr(Y_2) = \langle f_1, f_2 \rangle \tr(A_2A_1)=O(N^0)$, and induction on $ r $, using again equation (\ref{eq:12})  gives the result for  $ r \geq 3 $. 
\end{proof}

\bibliographystyle{alpha}

\end{document}